\documentclass[smallcondensed]{svjour3}

\usepackage{tikz-cd}
\smartqed  % flush right qed marks, e.g. at end of proof
\usetikzlibrary{decorations.markings,intersections}

\tikzset{%
scalearrow/.style n args={3}{
  decoration={
    markings,
    mark=at position (1-#1)/2*\pgfdecoratedpathlength
      with {\coordinate (#2);},
    mark=at position (1+#1)/2*\pgfdecoratedpathlength
      with {\coordinate (#3);},
    },
  postaction=decorate,
  }
}

\usepackage[utf8]{inputenc}
\usepackage[T1]{fontenc}
\usepackage{todonotes}
\usepackage{amsfonts}
\usepackage{amsmath}
\usepackage{mathtools}
\usepackage{txfonts}
\usepackage{ifpdf}
\ifpdf
\usepackage[all,pdf,2cell]{xy}\UseAllTwocells\SilentMatrices
\else
\usepackage[all,xdvi,2cell]{xy}\UseAllTwocells\SilentMatrices
\fi
\newcommand{\C}{\mathcal{C}}
\newcommand{\Rel}{\mathbf{Rel}}
\newcommand{\Pos}{\mathbf{Pos}}
\newcommand{\Cat}{\mathbf{Cat}}
\newcommand{\Endo}{\mathbf{Endo}}
\newcommand{\Mnd}{\mathbf{Mnd}}
\newcommand{\id}{\mathrm{id}}
\newcommand{\Span}{\mathrm{Span}}
\newcommand{\dg}{\delta}
\newcommand{\RelMon}{\mathbf{RelMon}}
\newcommand{\Set}{\mathbf{Set}}
\newcommand{\up}{\nearrow}
\renewcommand{\*}{\xmapsto{\,*\,}}
\newcommand{\isleftadjoint}{\dashv}

\newtheorem{fact}{Fact}
\begin{document}
\title{On monoids in the category of sets and relations}
\author{Anna Jen\v cov\'a 
% \thanks{Mathematical Institute, Slovak Academy of Sciences}
\and Gejza Jen\v ca
%\thanks{
%Department of Mathematics and Descriptive Geometry,
%Faculty of Civil Engineering,
%Slovak University of Technology
%}
}
\institute{
A. Jenčová \at
Mathematical Institute, Slovak Academy of Sciences,
Slovak Republic\\
		\email{jencova@mat.savba.sk}
\and
G. Jenča \at
Department of Mathematics and Descriptive Geometry\\
Faculty of Civil Engineering,
Slovak University of Technology,
	Slovak Republic\\
              \email{gejza.jenca@stuba.sk}
}
\maketitle 
\begin{abstract}
The category $\mathbf{Rel}$ is the category of sets (objects) and relations (morphisms). Equipped
with the direct product of sets, $\mathbf{Rel}$ is a monoidal category. Moreover,
$\mathbf{Rel}$ is a locally posetal 2-category, since every homset $\mathbf{Rel}(A,B)$ is a poset with
respect to inclusion. We examine the 2-category of monoids $\mathbf{RelMon}$
in this category. The morphism we use are lax.

This category includes, as subcategories, various interesting classes:
hypergroups, partial monoids (which include various types of quantum
logics, for example effect algebras) and small categories. We show how
the 2-categorical structure gives rise to several previously defined
notions in these categories, for example certain types of congruence relations 
on generalized effect
algebras. This explains where these definitions come from. 
\end{abstract}
\begin{acknowledgements}
This research is supported by grants VEGA 2/0069/16, 1/0420/15,
Slovakia and by the Slovak Research and Development Agency under the contract
APVV-14-0013.
\end{acknowledgements}

\section{Introduction}

A {\em strict 2-category} \cite{ehresmann1963categories,benabou1965categories}
is a category with ``morphisms between morphisms'' or, in other words,
a category where the set of all homomorphism between two objects carries
the structure of a category. The most important example is $\Cat$ -- the category of small
categories.

The most straightforward definition of a strict 2-category is relatively
simple: it is a category enriched \cite{kelly1982basic} in the (cartesian closed) category $\Cat$.
Unfortunately, this definition is not general enough to cover many interesting
cases, because it may happen that the composition of 1-cells is associative
only up to a 2-isomorphism (for example, the category of spans over $\Set$ is
not strict), so one has to weaken the axioms, obtaining a notion of a {\em weak
2-category}, sometimes called a {\em bicategory}.
We refer the reader to
\cite{lack20102categories,kelly1974review,leinster2004higher} for an introduction to the subject
of 2-categories and to \cite{mac1998categories,awodey2006category} for category-theoretical
terminology.

The starting point of the {\em formal category theory} is the
observation that one can formulate various categorical notions (for example,
monads, adjunctions, Kan extensions) in the language of 2-categories. Changing the
underlying 2-category from $\Cat$ to some other 2-category $\C$, it may then 
happen that these notions give rise to either well-known or new notions, perhaps allowing
for a new insight. Let us illustrate this phenomenon by an example:

\begin{example}
For every category $\C$, there is a {\em bicategory of spans $\Span(\C)$}
(see \cite[Chapter XII, Section 7]{mac1998categories}). If $\C$ is $\Set$ then the monads in $\Span(\C)$ are
small categories. If $\C$ is the category of groups then the monads in $\Span(\C)$ are the
{\em twisted modules} (see \cite[Chapter XII, Section 8]{mac1998categories}).
\end{example}

The aim of this paper is to examine the notions ``adjunction'' and ``monad'' in
the 2-category of monoids in the monoidal 2-category of sets and relations,
equipped with the direct product of sets $(\Rel,\times,1)$. We denote
this 2-category of monoids by $\RelMon$. 
This category includes partial monoids (which include effect algebras), as well as
small categories (considered as sets of arrows equipped with the composition).

Realistically, one probably cannot hope to achieve some sort of 
``real result'' from these considerations. However, we find it interesting
and surprising that some notions and conditions used in quantum logics
appear to come from monads and adjunctions in $\RelMon$. Moreover, there
are other concrete manifestations of these abstract notions in other parts of mathematics,
as demonstrated by several examples.

Recently, there were several other papers published in the area of categorical
quantum mechanics \cite{abramsky2009categorical} that concern $\Rel$ and $\RelMon$. In \cite{heunen2013relative} and
\cite{heunen2015categories}, authors establish an interesting 
equivalence between {\em special dagger Frobenius structures}
in the dagger monoidal category $\Rel(\mathcal C)$ and internal groupoids in $\mathcal C$,
for a regular category $\mathcal C$. In \cite{contreras2015groupoids}, 
the results from \cite{heunen2013relative} are extended to describe a correspondence between
certain types of generalized groupoids and associative structures in $\Rel$, establishing a link 
between these abstract results and Poisson sigma models.
In \cite{heunen2016monads}, monads on dagger categories are investigated.
In \cite{pavlovic2017modular}, effect algebras are characterized as certain monoids
in $\Rel$, using merely the dagger-compact structure of $\Rel$.

\section{The 2-category of sets and relations}

In this section, we review some elementary facts concerning the 2-category of sets
and relations. Everything in this section is well-known, see \cite{benabou1967introduction}.

The {\em category of sets and relations}, denoted by $\Rel$, is a category 
whose objects (or $0$-cells) are sets and arrows (or $1$-cells) are relations 
$f\subseteq A\times B$. The composite of arrows $f\colon A\to B$ and 
$g\colon B\to C$ is the arrow $(g\circ f)\colon A\to C$ given by the rule
$$
(a,c)\in (g\circ f)\Leftrightarrow (\exists b\in B)(a,b)\in f\text{ and }(b,c)\in g.
$$
The identity arrow $\id_A\colon A\to A$ is the identity relation $\id_A=\{(a,a)\colon a\in A\}$.

Note that there is an obvious faithful functor 
$U\colon\Set\to\Rel$ that is identity on objects and takes a mapping $f\colon A\to B$ to
its graph
$$
\{(x,y)\in A\times B\colon f(x)=y\}.
$$
This forgetful functor is a left adjoint, the corresponding right adjoint is the powerset/image functor
$P\colon \Rel\to\Set$. This adjunction induces the well-known {\em covariant powerset monad} on
$\Set$.  $\Rel$ is then isomorphic to the Kleisli category for this monad.

Moreover, the category of sets and relations is a 2-category: if $h_1,h_2$ are relations $A\to B$, then a
2-cell $h_1\to h_2$ is simply the fact that $h_1\subseteq h_2$. Thus, every hom-category
in $\Rel$ is a poset.

As usual, we draw a 2-cell in a commutative diagram as a double arrow, for example 
\begin{equation}
\label{cd:example}
\begin{tikzcd}
A
	\ar[r,"f_2"]
	\ar[d,"f_1"']
&
B
	\ar[d,"g_2"]
\\
C
	\ar[r,"g_1"']
	\ar[ru,phantom,scalearrow={0.333}{start}{end}]
	\ar[ru,Rightarrow,to path = (start)--(end)]
&
D
\end{tikzcd}
\end{equation}

means that $g_1\circ f_1\subseteq g_2\circ f_2$. Note that on the level of
elements, the diagram (\ref{cd:example}) means that
\begin{itemize}
\item for every $a\in A$ and $d\in D$ such that there is a $c\in C$ with $(a,c)\in f_1$ 
and $(c,d)\in g_1$,
\item there exists $b\in B$ such that $(a,b)\in f_2$ and $(b,d)\in g_2$.
\end{itemize}

Besides the structure of a 2-category, $\Rel$ carries the structure of a 
{\em dagger category}: there is an involution functor $\dag\colon \Rel\to\Rel^{op}$ that is identity on objects.
For a relation $f\subseteq A\times B$, $f^{\dag}\subseteq B\times A$ is the relation
given by the equivalence 
$$
(b,a)\in f^{\dag}\Leftrightarrow (a,b)\in f.
$$

If $A,B\in \Rel$, then the disjoint union of sets $A \sqcup B$ is both the product and the 
coproduct of $A,B$ in $\Rel$. Since $\Rel$ lacks some (co)equalizers,  it is not a (co)complete category.

Considering $\Rel$ as a 2-category, we may look at various category-theoretic notions
in $\Rel$.

Recall \cite{lack20102categories}, that in a 2-category a 1-cell $f\colon A\to B$ is left adjoint to a 1-cell $g\colon B\to A$ if and only if there
are 2-cells $\eta\colon \id_A\to g\circ f$ and $\varepsilon\colon f\circ g\to\id_B$ such that in the hom-categories
$[A,B]$ and $[B,A]$ the diagrams
$$
\begin{tikzcd}
f
	\ar[r,"f\eta"]
	\ar[rd,"\id_f"']
&
fgf
	\ar[d,"\varepsilon g"]
\\
~
&
f
\end{tikzcd}
\qquad
\begin{tikzcd}
g
	\ar[r,"\eta f"]
	\ar[rd,"\id_g"']
&
gfg
	\ar[d,"g\varepsilon"]
\\
~
&
g
\end{tikzcd}
$$
commute.

However, since every hom-category in $\Rel$ is a poset, these conditions are automatically valid
whenever there exist 2-cells $\id_A\subseteq gf$ and $fg\subseteq\id_B$. A straightforward
reasoning gives us the following fact.

\begin{fact} 
An arrow $f\colon A\to B$ in $\Rel$ is a left adjoint to an arrow $g\colon B\to A$ 
if and only if $f$ is (a graph of) a mapping $A\to B$ and $g=f^{\dag}$.
\end{fact}

Note that this implies that the canonical inclusion
$\Set\to\Rel$ embeds $\Set$ into $\Rel$ as the subcategory of left-adjoints in
the 2-category $\Rel$.

Recall \cite{street1972formal}, that a {\em monad} in a 2-category is an object $A$ equipped
with a triple $(s,\eta,\mu)$, where $s\colon A\to A$, $\eta\colon \id_A\to s$ and $\mu\colon s\circ s\to s$
such that in the hom-category $[A,A]$ the equations
$\mu\circ s\eta=\mu\circ\eta s=\id_s$ and $\mu\circ s\mu=\mu\circ\mu s$ hold.

Similarly as for the notion of a left-adjoint, the fact that $\Rel$ is enriched
in $\Pos$ implies that these equations for $\eta$ and $\mu$ are valid whenever
$\eta$ and $\mu$ exist. Thus, an $s\colon A\to A$ in $\Rel$ is an underlying 1-cell
of a monad if and only if $\id_A\subseteq s$ and $s\circ s\subseteq s$.
In other words,

\begin{fact}
Monads in $\Rel$ are preorders. 
\end{fact}

Indeed, observe that $\id_A\subseteq s$ means that $s$ is a reflexive and $s\circ s\subseteq s$ means
that $s$ is transitive.

In a 2-category, if $f$ is left adjoint to $g$ (in symbols $f\isleftadjoint g$), then the quadruple
$(f,g,\eta,\varepsilon)$ gives rise to a monad $(gf,\eta,g\varepsilon f)$ on the domain of $f$.

In the 2-category $\Cat$, every monad arises from an adjunction. This is not true in $\Rel$.

\begin{fact}
	A monad $s\colon A\to A$ in $\Rel$ arises from an adjunction if and
only if $s$ is an equivalence relation.  
\end{fact}

Indeed, if $f\colon A\to B$ is a mapping (that means, a left adjoint in
$\Rel$), then the monad associated with the corresponding adjunction is
$f^{\dag}\circ f\colon A\to A$.  This is the equivalence relation on $A$ given by the
decomposition of $A$ to the fibers of $f$, usually called {\em the kernel of
$f$}. On the other hand, if $\sim$ is an equivalence relation on $A$, then we have
an obvious adjunction between $A$ and the quotient $A/\sim$ that in turn
gives rise to $\sim$.

\section{Monoids in $\Rel$}

It is easy to see that the cartesian product $\times$ of sets is a bifunctor from
$\Rel\times\Rel$ to $\Rel$. 

As $\times$ is the product in $\Set$, it satisfies
the coherence conditions for a monoidal category \cite[Chapter VII]{mac1998categories}, 
so $(\Rel,\times,1)$ is a monoidal category.

\begin{definition}
Let $(\C,\otimes,1)$ be a monoidal category. A {\em monoid} in $\C$ is
a triple $(A,e,*)$, where $A$ is an object of $\C$, $e\colon 1\to A$ and
$*:A\otimes A\to A$ such that the following diagrams commute
$$
\begin{tikzcd}[column sep=4em,row sep=3em]
A\otimes 1
	\ar[r,"\id\otimes e"]
	\ar[rd,"\rho"']
&
A\otimes A
	\ar[d,"*"]
&
1\otimes A
	\ar[l,"e\otimes\id"']
	\ar[ld,"\lambda"]
\\
~
&
A
\end{tikzcd}
\qquad
\begin{tikzcd}
A\otimes(A\otimes A)
	\ar[rr,"\alpha"]
	\ar[d,"1\otimes *"']
&
~
&
(A\otimes A)\otimes A
	\ar[d,"*\otimes\id"]
\\
A\otimes A
	\ar[rd,"*"']
&
~
&
A\otimes A
	\ar[ld,"*"]
\\
~
&
A
&
~
\end{tikzcd}
$$
Here, $\lambda,\rho$ and $\alpha$ denote the (left and right) unitors and the associator
of the monoidal category $\C$.
\end{definition}

The triangle diagrams are called the {\em right (left) unit axioms}. The pentagon diagram
is called the {\em associativity axiom}.

The monoids in the category $(\Rel,\times,1)$ are called {\em relational monoids}.

Let us spell out the axioms of a relational monoid in detail. Let $(A,e,*)$ be
a relational monoid. Since $e\colon 1\to A$ is a relation, we may identify $e$ with a
subset 
$$
E_A=\{y\in A\colon (1,y)\in e\}
$$
of $A$, which we call {\em the set of units of $A$}.

The $*$ is a relation from $A\times A$ to $A$, so it is a 
subset of $(A\times A)\times A$.
We shall write $(a_1,a_2)\*a$ to denote the fact that 
$((a_1,a_2),a)\in *\subseteq (A\times A)\times A$.

The right unit axiom means that, for every $a\in A$, there is
$y\in E_A$ such that $(a,y)\* a$ and, at the same time,
whenever there is a $y\in E_A$ such that $(a,y)\* b$, then $a=b$. The meaning of the left unit axiom is similar.

Associativity axiom means that for every quadruple $a_1,a_2,a_3,z$ of elements of $A$, the
following statements are equivalent:
\begin{itemize}
\item there exists $w\in A$ such that $(a_1,a_2)\* w$ and $(w,a_3)\* z$;
\item there exists $w'\in A$ such that $(a_2,a_3)\* w'$ and $(a_1,w')\* z$.
\end{itemize}

We know that every ordinary monoid $A$ in $\Set$ has exactly one unit. 
In general, this is not true for relational monoids.
\begin{proposition}
Let $A$ be a relational monoid. 
For every $a\in A$, there is exactly one $y\in E_A$ (called 
the right unit of $a$) such that $(a,y)\* a$.
\end{proposition}
\begin{proof}
By previous remarks, there exists $y\in E_A$ such that $(a,y)\* a$. 
Let us prove that this $y$ is unique.

Let $y'$ be another right
unit of $a$. We see that 
$$
((a,y),y')\xmapsto{*\times\id_A}(a,y')\* a.
$$
By the associativity axiom, there is some $z\in A$ such that
$$
(a,(y,y'))\xmapsto{\id_A\times *}(a,z)\* a
$$
So, in particular, $(y,y')\* z$ and $y'\in E_A$. Therefore, by the right unit axiom,
$y=z$. Similarly, by the left unit axiom, $y'=z$ and this implies $y=y'$.
\end{proof}
By a symmetrical argument, there is exactly one left unit for every element of $A$.

Let us consider some examples of relational monoids.
\begin{example}
Every ordinary monoid in $\Set$ is a relational monoid.
\end{example}
\begin{example}
Every hypergroup \cite{wall1937hypergroups} is a relational monoid.
\end{example}
\begin{example}
\label{ex:catismonoid}
Every small category is a relational monoid: the underlying set of the
relational monoid corresponding to a category $C$ is the set all arrows
in $C$. Multiplication is the composition of arrows and the set of units is the
set of all identity arrows of $C$.
This observation goes back to the seminal paper
\cite{barr1970relational}, see also \cite{kenney2011categories,heunen2013relative}
for more results on the connections between $\RelMon$ and $\Cat$.
\end{example}
\begin{example}
\label{ex:posetismonoid}
As a consequence of the previous example, the set of all comparable
pairs in a poset is a relational monoid.

Explicitly, let $(A,\leq)$ be a poset, write $Q(A)$ for the set of all comparable pairs of elements of $A$.
In lattice theory, the elements of $Q(A)$ are called {\em quotients}. 
As usual (see for example \cite{Gra:GLT}) we write $b/a\in Q(A)$ to express the facts that 
that $a,b\in A$ and $a\leq b$.

Let us equip $Q(A)$ with the relation $*\colon Q(A)\times Q(A)\to Q(A)$ given by the
rule $(b/a,d/c)\* (d/a)$ if and only if $b=c$ and the relation $e\colon 1\to Q(A)$ that
selects the trivial quotients of the type $a/a$.

Then $(Q(A),*,e)$ is a relational monoid.
\end{example}
\begin{example}
Let $\mathbb R_0^+$ be the set of all nonnegative real numbers, 
let $*:\mathbb R_0^+\times\mathbb R_0^+\to\mathbb R_0^+$ be a
relation given by the rule $(a,b)\* x$ if and only if $a\leq x\leq a+b$ and
let $e:1\to\mathbb R_0^+$ be a relation that picks out $0$ from $\mathbb R_0^+$. Then
$(\mathbb R_0^+,*,e)$ is a relational monoid. Note that $*$ is not a partial mapping.
\end{example}

For every monoidal category $(C,\otimes,1)$, the class of monoids in $C$
comes equipped with a standard notion of morphism between monoids,
giving rise to a category of monoids in $C$.
However, this notion does not work in
examples we are interested in. It turns out that another notion is
more appropriate for our purposes.

For relational monoids $A,B$ and a relation $h\colon A\to B$, we
say that $h$ is a {\em morphism of relational monoids} if and only if there
are 2-cells
$$
\begin{tikzcd}
A\times A
	\ar[r,"h\times h"]
	\ar[d,"*"']
&
B\times B
	\ar[d,"*"]
\\
A
	\ar[ru,phantom,scalearrow={0.333}{start}{end}]
	\ar[ru,Rightarrow,to path = (start)--(end)]
	\ar[r,"h"']
&
B
\end{tikzcd}
\qquad
\begin{tikzcd}
1
	\ar[r,"e"]
	\ar[rd,"e"'{name=U}]
&
A
	\ar[d,"h"]
	\ar[Leftarrow,from=U,shorten >= 2pt , shorten <= 5pt]
\\
~
&
B
\end{tikzcd}
$$

By a {\em category of relational monoids} we mean a 2-category in which
\begin{itemize}
\item 0-cells are relational monoids,
\item 1-cells are morphisms of relational monoids,
\item 2-cells are the inclusions of relations, inherited from $\Rel$.
\end{itemize}
The category of relational monoids is denoted by $\RelMon$. 

\begin{example}
The power set $P(\mathbb N^+)$ of the set of all positive natural numbers,
equipped with a elementwise multiplication, is a monoid with
a neutral element $\{1\}$. 
Let us define a relation $h\colon P(\mathbb N^+)\to\mathbb N$,
where $\mathbb N$ is the additive monoid of natural numbers,
by the rule $(X,n)\in h$ if and only if there is some $a\in X$ such that
the length of the prime decomposition of $a$ is equal to $n$.
Then $h$ is a morphism in $\RelMon$ from $P(\mathbb N^+)$ to $(\mathbb N,+,0)$
that is not a graph of mapping.
\end{example}

Since $\RelMon$ is a 2-category, we may consider adjunctions in $\RelMon$.
Let $A,B$ be relational monoids, let $f\colon A\to B$ and $g\colon B\to A$ be morphisms
in $\RelMon$. Then it is easy to check
that $f$ is left adjoint to $g$ if and only if $f$ is a mapping and
$g=f^\dag$.

From this, we obtain a characterization of left adjoints in $\RelMon$.
\begin{proposition}\label{prop:leftadjoint}
A morphism $f\colon A\to B$ of relational monoids is a left adjoint if and only if
$f$ is a mapping and the following conditions are satisfied.
\begin{enumerate}
\item[(L1)] For all $b_1,b_2\in B$ and $a\in A$ such that $(b_1,b_2)\* f(a)$
there exist $a_1,a_2\in A$ such that $f(a_1)=b_1$, $f(a_2)=b_2$ and $(a_1,a_2)\* a$.
\item[(L2)] If $x\in A$ and $f(x)\in E_B$, then $x\in E_A$.
\end{enumerate}
\end{proposition}
\begin{proof}
Clearly, a morphism of relational monoids 
$f$ is left adjoint in $\RelMon$ if and only if $f$ is left adjoint
in $\Rel$ (that means, a mapping) and
$f^\dag$ is a morphism in $\RelMon$.
It remains to observe that the conditions (L1) and (L2) just spell out
that the right adjoint $f^\dag$ is a morphism of relational monoids.
\end{proof}

\begin{example}
Let $K$ be a field.  Let $K_m(X)$ be set of all monic polynomials over $K$
equipped with the multiplication of polynomials. Then $K_m(X)$ is an ordinary
monoid in $\Set$, hence it is a relational monoid.  Consider the mapping
$\dg\colon K_m(X)\to\mathbb N$ that takes every polynomial to its degree. Then $\dg$
is a morphism of monoids.  Moreover, $\dg$ is a left adjoint in $\RelMon$ if
and only if $K$ is algebraically closed.

Indeed, let $\dg$ be a left adjoint in $\RelMon$ and let $p$ be a monic polynomial of degree
greater than 1. Since we have $\dg(p)=1+(\dg(p)-1)$, property
(L1) of Proposition \ref{prop:leftadjoint} implies that there
are $p_1,p_2\in K_m(X)$ such that $\dg(p_1)=1$, $\dg(p_2)=\dg(p)-1$ and
$p=p_1.p_2$. So $p$ is divisible by a polynomial of degree $1$, hence $p$ has a
root.

Assume that $K$ is algebraically closed. Let us prove (L1), (L2) of Proposition
\ref{prop:leftadjoint}. Let $p\in K_m(X)$ and suppose that $\dg(p)=n_1+n_2$. To
prove (L1), we need to find monic polynomials such that $p=p_1.p_2$,
$\dg(p_1)=n_1$ and $\dg(p_2)=n_2$. This is easy, because $p$ is a product of
some polynomials of degree $1$. Moreover, $\dg(p)=0$ if and only if $p=1$ (this
is why we have to consider monic polynomials). So (L2) holds and hence $\dg$ is
left adjoint in $\RelMon$.  
\end{example}

\section{Monads in $\RelMon$}

A monad in the 2-category $\RelMon$ on a relational monoid $(A,*,e)$ 
is necessarily a monad in $\Rel$ on the underlying set $A$. Thus a monad on
$(A,*,e)$ is a preorder on the set $A$ which is, at the same time, an endomorphism
of the relational monoid $A$.
$$
\begin{tikzcd}
A\times A
	\ar[r,"\leq\times \leq"]
	\ar[d,"*"']
&
A\times A
	\ar[d,"*"]
\\
A
	\ar[ru,phantom,scalearrow={0.333}{start}{end}]
	\ar[ru,Rightarrow,to path = (start)--(end)]
	\ar[r,"\leq"']
&
A
\end{tikzcd}
\qquad
\begin{tikzcd}
1
	\ar[r,"e"]
	\ar[rd,"e"'{name=U}]
&
A
	\ar[d,"\leq"]
	\ar[Leftarrow,from=U,shorten >= 2pt , shorten <= 5pt]
\\
~
&
A
\end{tikzcd}
$$

Explicitly, a preorder $\leq$ on $A$ is a monad in $\RelMon$ if and only if
for all $a_1,a_2,a,a'\in A$
such that $(a_1,a_2)\* a\leq a'$, there are $a'_1,a'_2\in A$ such that $a_1\leq a'_1$,
$a_2\leq a'_2$ and $(a_1',a_2')\* a'$, moreover, for every $y\in E_A$, $y\leq x$ implies that $x\in E_A$.
\footnote{A reader who knows what the {\em Riesz decomposition property} means might wish to look at 
Example \ref{ex:RDP} now.} 
Let us look at some examples of monads in in $\RelMon$.

\begin{example}\label{ex:divides}
Consider the monoid $(\mathbb N,+,0)$. Equip $\mathbb N$ with the divisibility
partial order $\mid$, meaning that $a\mid a'$ if and only if there is $b\in\mathbb N$ such
that $ab=a'$. Assume that $a_1+a_2=a\mid a'$.  Then there is $b$ such that
$(a_1+a_2)b=a'$ and, putting $a'_1=a_1b$, $a'_2=a_2b$ we see that $a_1\mid
a'_1$, $a_2\mid a'_2$ and $a'_1+a'_2=a'$. Moreover $0\mid x$ implies that
$x=0$.  Therefore, $\mid$ is a monad on $\mathbb N$.
\end{example}

\begin{example}
Let $\Sigma$ be a set.  Consider the free monoid $\Sigma^*$, consisting
of all words over the alphabet $\Sigma$, equipped with the concatenation of words.
Recall, that a word $y$ is a subword of a word $x$ if we can obtain $y$
from $x$ by deleting the letters at some positions in $x$. For example,
the word $abc$ is a subword of the word $cacbacab$.
For $x,y\in\Sigma^*$ write $x\geq y$ if and only if $y$ is a subword of $x$. Then
$\geq$ is a monad on $\Sigma^*$. 

Indeed, if $y$ is a subword of $x1.x2$, then
$y=y_1.y_2$, where $y_1$ is a subword of $x_1$ and $y_2$ is a subword of $x_2$.
Moreover, $x$ is a subword of the empty word if and only if $x$ is empty.
Therefore, $\geq$ is a monad on the free monoid.
\end{example}

Let $\Endo(\RelMon)$ be a category, in which
\begin{itemize}
\item objects are all pairs $(A,f)$, where $f$ is an endomorphism $f\colon A\to A$ in $\RelMon$ 
\item a morphism $v\colon (A,f)\to(B,g)$ is an oplax commutative square
$$
\begin{tikzcd}
A
	\ar[r,"f"]
	\ar[d,"v"']
&
A
	\ar[d,"v"]
	\ar[ld,phantom,scalearrow={0.333}{start}{end}]
	\ar[ld,Rightarrow,to path = (start)--(end)]
\\
B
	\ar[r,"g"']
&
B
\end{tikzcd}
$$
where $v$ is a morphism of relational monoids.
\end{itemize}
We write $\Mnd(\RelMon)$ for the full subcategory of monads in $\Endo(\RelMon)$.
\begin{lemma}
Let $A,B$ be relational monoids, let $(f_i)_{i\in I}$ be a family of morphisms
with $f_i\colon A\to B$. Then the relation
$f=\bigcup_{i\in I}f_i\colon A\to B$ is a morphism of relational monoids.
\end{lemma}
\begin{proof}
Trivial.
\end{proof}

\begin{theorem}
$\Mnd(\RelMon)$ is a reflexive subcategory of $\Endo(\RelMon)$.
\end{theorem}
\begin{proof}
Let $(A,f)$ be an object of $\Endo(\RelMon))$.
Write $cl(f)$ for the reflexive and transitive closure of the relation $f$.
As $cl(f)=\bigcup_{i=0}^\infty f^i$ is a union of a family of morphisms,
$cl(f)$ is an endomorphism of $A$, so $(A,cl(f))$ is an object of $\Endo(\RelMon)$.
Moreover, since $cl(f)$ is a preorder, $(A,cl(f))$ is an object of $\Mnd(\RelMon)$.
We claim that the morphism 
$$
\begin{tikzcd}
A
	\ar[r,"f"]
	\ar[d,"\id_A"']
&
A
	\ar[d,"\id_A"]
	\ar[ld,phantom,scalearrow={0.333}{start}{end}]
	\ar[ld,Rightarrow,to path = (start)--(end)]
\\
A
	\ar[r,"cl(f)"]
&
A
\end{tikzcd}
$$
is a reflection, that means, for every object $(B,\leq)$ of $\Mnd(\RelMon)$ and for every
arrow $u\colon (A,f)\to (B\leq)$ there is unique dotted arrow such that
$$
\begin{tikzcd}
(A,f)
	\ar[rd,"u"]
	\ar[d,"\id_A"']
\\
(A,cl(f))
	\ar[r,dotted]
&
(B,\leq)
\end{tikzcd}
$$
commutes.
Note that, if the dotted arrow exists, then it must be induced by $u$. So it
suffices to prove that $u$ induces a morphism in $\Endo(\RelMon)$ from
$(A,cl(f))$ to $(B,\leq)$.

We claim that, for all $n\in\mathbb N$, $u$ induces a morphism in $\Endo(\RelMon)$
from $(A,f^n)$ to $(B,\leq)$. For $n=0$ this is trivial. Suppose that our
claim is valid for $n=k$. Pasting together the 2-cells
$$
\begin{tikzcd}
A
	\ar[r,"f^k"]
	\ar[d,"u"']
&
A
	\ar[d,"u"]
	\ar[r,"f"]
	\ar[ld,phantom,scalearrow={0.333}{start}{end}]
	\ar[ld,Rightarrow,to path = (start)--(end)]
&
A
	\ar[d,"u"]
	\ar[ld,phantom,scalearrow={0.333}{startt}{endt}]
	\ar[ld,Rightarrow,to path = (startt)--(endt)]
\\
B
	\ar[r,"\leq"']
&
B
	\ar[r,"\leq"']
&
B
\end{tikzcd}
$$
gives us the 2-cell
$$
\begin{tikzcd}
A
	\ar[r,"f^{(k+1)}"]
	\ar[d,"u"']
&
A
	\ar[d,"u"]
	\ar[ld,phantom,scalearrow={0.333}{start}{end}]
	\ar[ld,Rightarrow,to path = (start)--(end)]
\\
B
	\ar[r,"\leq"']
&
B
\end{tikzcd}
$$
Thus, for all $n\in\mathbb N$, $(u\circ f^n)\subseteq(\leq\circ u)$. Taking the union
of these inclusions over $n\in\mathbb N$ gives us the inclusion
$(u\circ cl(f))\subseteq(\leq\circ u)$, meaning that $u$ induces a morphism
in $\Endo(\RelMon)$.
\end{proof}
Thus, every endomorphism in $\RelMon$ generates a monad in $\RelMon$.
\begin{example}
Consider the monoid $(\mathbb N,+,0)$, fix $k\in\mathbb N\setminus\{0\}$ 
and the endomorphism $f_k\colon \mathbb N\to\mathbb N$ given by
$f_k(a)=ka$. The reflection of the object $(\mathbb N,f_k)$ of $\Endo(\RelMon)$ is a monad $(\mathbb N,\leq_k)$,
where the preorder $\leq_k$ is given by the rule 
$a\leq_k b$ if and only if $a\mid b$ and $b/a$ is a power of $k$. 
\end{example}

\section{Modular lattices as monads in $\RelMon$}

We have seen (Example \ref{ex:posetismonoid}), that for every poset 
the set of all quotients $Q(A)$ is a relational monoid.
Let $A$ be a lattice.
There is a canonical partial order $\up$ on $Q(A)$ given by the rule $b/a\up d/c$ if and only if
$a=b\wedge c$ and $d=b\vee c$. This partial order plays a central role in the theory of
lattice congruences  (see \cite{Gra:GLT}).

Recall, that a lattice is {\em modular} if and only if, for all $x\leq y$, $y\wedge(x\vee z)=x\vee(y\wedge z)$.

\begin{proposition}
Let $A$ be a lattice. Then $(Q(A),\up)$ is a monad in $\RelMon$ if and only if $A$ is a modular lattice.
\end{proposition}
\begin{proof}
The statement that $(A,\up)$ is a monad means that the diagrams
$$
\begin{tikzcd}
1
	\ar[r,"e"]
	\ar[rd,"e"'{name=U}]
&
Q(A)
	\ar[d,"\up"]
	\ar[Leftarrow,from=U,shorten >= 2pt , shorten <= 5pt]
\\
~
&
B
\end{tikzcd}
\qquad
\begin{tikzcd}
Q(A)\times Q(A)
	\ar[r,"\up\times\up"]
	\ar[d,"\circ"]&
Q(A)\times Q(A)
	\ar[d,"\circ"]
\\
Q(A)
	\ar[r,"\up"]
	\ar[ru,phantom,scalearrow={0.25}{start}{end}]
	\ar[ru,Rightarrow,to path = (start)--(end)]
	&
Q(A)
\end{tikzcd}
$$
commute. The commutativity of the triangle diagram means that $a/a\up c/b$ implies
that $b=c$. This is easily seen to be true for every lattice $A$.

The commutativity of the square is equivalent to the following property of the lattice $A$:

(**) For every $b/a,c/b,c'/a'\in Q(A)$ such that 
$(b/a)\circ(c/b)=c/a\up c'/a'$ there exists $b'\in A$ such that $a'\leq b'\leq c'$ and 
$b/a\up b'/a'$, $c/b\up c'/b'$.

Let us prove that the modularity of $A$ implies the property (**). Suppose that $A$ is a
modular lattice and let $a,b,c,a',c'$ be as in the assumption of (**). Let us put $b'=b\vee a'$
so that $b/a\up b'/a'$.
We claim that $c/b\up c'/b'$, that means, $c\vee b'=c'$, $c\wedge b'=b$.
Since $c/a\up c'/a'$, we see that
$$
c\vee b'=c\vee b\vee a'=c\vee a'=c'.
$$
and, applying the modular law with $b\leq c$, we obtain
$$
c\wedge b'=c\wedge(b\vee a')=b\vee(c\wedge a')=b\vee a=b,
$$
which means that $c/b\up c'/b'$.

Suppose that $A$ is a lattice satisfying the (**) property. Let $x,y,z\in A$ be such that
$x\leq y$. We need to prove that $y\wedge(x\vee z)=x\vee(y\wedge z)$. Put $a=y\wedge z$,
$b=x\vee(y\wedge z)$, $c=y$, $a'=z$, $c'=y\vee z$. We see that $a,b,c,a',c'$ satisfy the
assumptions of (**), hence there is a $b'$ such that 
$a'\leq b'\leq c'$ and $b/a\up b'/a'$, $c/b\up c'/b'$.
This implies that $b=c\wedge b'=c\wedge(b\vee a')$ and therefore
$$
x\vee(y\wedge z)=y\wedge (x\vee (y\wedge z)\vee z)=y\wedge(x\vee z).
$$
\end{proof}
For modular lattices $A$ and $B$ and a lattice morphism
$v:A\to B$, we write $Q(v):Q(A)\to Q(B)$ for the mapping
given by the rule $Q(v)(a/b)=v(a)/v(b)$. 
\begin{corollary}
$Q$ is a functor from the category of modular lattices to the category $\Mnd(\RelMon)$.
\end{corollary}
\begin{proof}
The proof is straightforward and is thus omitted.
\end{proof}

\section{Quantum structures as relational monoids}

Let $(P,+,0)$ be a partial algebra with a nullary operation $0$
and a binary partial operation $+$. Denote the domain of $+$ by
$\perp$. $P$ is called a {\em partial abelian monoid}
if and only if for all $a,b,c\in P$ the following conditions are satisfied:
\begin{enumerate}
\item[(P1)] $b\perp c$ and $a\perp b+c$ implies
        $a\perp b$, $a+b\perp c$, $a+(b+c)=(a+b)+c$.
\item[(P2)] $a\perp b$ implies $b\perp a$ and $a+b=b+a$
\item[(P3)] $a\perp 0$ and $a+0=a$.
\end{enumerate}
A partial abelian monoid $P$ 
is {\em positive} if and only if, for all $a,b\in P$, $a+b=0$ implies $a=b=0$.
A partial abelian monoid is {\em cancellative} if and only if, for all $a,b,c\in P$,
$a+c=a+b$ implies $b=c$. A cancellative and positive partial abelian monoid
is called a {\em generalized effect algebra}.

On every generalized effect algebra, there is a canonical partial order given by the rule
$a\leq c$ if and only if there is $b$ such that $a+b=c$. A generalized effect
algebra that is upper bounded is an {\em effect algebra}. 
Effect algebras were introduced in \cite{FouBen:EAaUQL}, the 
definition we give here is different but equivalent with the original one. See also
\cite{KopCho:DP} and \cite{GiuGre:TaFLfUP} for other axiomatizations of effect
algebras.

The prototype effect
algebra is $(\mathcal E(\mathbb H),\oplus,0,I)$, where $\mathbb H$ is a
Hilbert space and $\mathcal E(\mathbb H)$ consists of all self-adjoint
operators $A$ of $\mathbb H$ such that $0\leq A\leq I$. For 
$A,B\in\mathcal E(\mathbb H)$, $A\oplus B$ is defined iff $A+B\leq I$ and then
$A\oplus B=A+B$. The set $\mathcal E(\mathbb H)$ plays an important role in the
foundations of quantum mechanics \cite{Lud:FoQM}, \cite{BusGraLah:OQP}.

It is obvious that every generalized effect algebra is a monoid in $\RelMon$.
Let $A,B$ be generalized effect algebras.  A mapping $f\colon A\to B$ is
a {\em morphism of generalized effect algebras} if and only if $f(0)=0$ and for all
$x,y\in A$ such that $x\perp y$ we have $f(x)\perp f(y)$ and $f(x+y)=f(x)+f(y)$.
Note that every morphism of generalized effect 
algebras is a morphism in $\RelMon$. Thus,
the category of generalized effect algebras is a subcategory of $\RelMon$.

\begin{example}\label{ex:RDP}
Let $(E,+,0)$ be a generalized effect algebra. What does it mean that the canonical
partial order $\geq$ is a monad in $\RelMon$ on $E$?
The square diagram means that, for all $x_1,x_2,y\in E$, $x_1+x_2\geq y$ implies that
there are $y_1,y_2\in E$ such that $x_1\geq y_1$, $x_2\geq y_2$ and $y=y_1+y_2$. This
is a well-known condition, called the {\em Riesz decomposition property}
\cite{Goo:POAGwI,JenPul:QoPAMatRDP}.
The triangle diagram means that $0\geq x$ implies that $x=0$, which is true in
any generalized effect algebra. Thus, $\geq$ is a
monad on a generalized effect algebra if and only if the generalized effect
algebra satisfies the Riesz decomposition property.
\end{example}

Similarly as in $\Rel$, a monad $(A,\leq)$ in $\RelMon$ arises from an adjunction if and only if
the preorder $\leq$ is an equivalence.

Explicitly, this gives us the following conditions:
\begin{enumerate}
\item[(M1)] $\sim$ is an equivalence.
\item[(M2)] The diagram
$$
\begin{tikzcd}
A\times A
	\ar[r,"\sim\times\sim"]
	\ar[d,"*"]
&
A\times A
	\ar[d,"*"]
\\
A
	\ar[ru,phantom,scalearrow={0.333}{start}{end}]
	\ar[ru,Rightarrow,to path = (start)--(end)]
	\ar[r,"\sim"]
&
A
\end{tikzcd}
$$
commutes.
\item[(M3)] If $x\sim y$ and $y$ is a unit of $A$, then $x$ is a unit of $A$.
\end{enumerate}

\begin{proposition}\label{prop:congruence}
\cite{ChePul:SILiPAM}
Let $(A,+,0)$ be a partial abelian monoid. Let $\sim\subseteq A\times A$ be 
a relation satisfying the following:
\begin{enumerate}
\item[(C1)]$\sim$ is an equivalence relation.
\item[(C2)]If $x_1+y_1$ exists, $x_2+y_2$ exists, $x_1\sim x_2$ and $y_1\sim y_2$,
then $x_1+y_1\sim x_2+y_2$.
\item[(C5)]If $x+y$ exists and $(x+y)\sim z$, then there are $x_1,y_1\in A$ such that
$x_1\sim x$, $y_1\sim y$ and $x_1+y_1=z$.
\footnote{The notation (C1), (C2) and (C5) is inherited from the original paper \cite{ChePul:SILiPAM}. 
It coincides with the notation used later in several other papers and in the book \cite{DvuPul:NTiQS}.}
\end{enumerate}
Define a partial operation $+$ on the quotient $A/\sim$ by the rule
$[x]_\sim+[y]_\sim=[x_1+y_1]_\sim$, where $x_1,y_1\in A$ are such that 
$x_1\sim x$, $y_1\sim y$ and $x_1+y_1$ exists.
Then $+$ is well defined on $A/\sim$ and 
$(A/\sim,+,[0]_\sim)$ is a partial abelian monoid.
\end{proposition}
Let us note that the conditions from Proposition \ref{prop:congruence}
are not necessary for an equivalence to induce a partial abelian monoid structure
on $A/\sim$, they are merely sufficient.
\begin{proposition}
Let $(A,+,0)$ and $(B,+,0)$ be a partial abelian monoids, let $f\colon A\to B$ be a
left adjoint in $\RelMon$. Then the monad on $A$ arising from the adjunction
$f\isleftadjoint f^\dagger$ satisfies the conditions in Proposition \ref{prop:congruence}.
\end{proposition}
\begin{proof}
The monad $\sim:=f^\dag\circ f$ is an equivalence, so (C1) is satisfied.

Let $x_1,x_2,y_1,y_2$ be as in the assumptions of (C2). In this context
that means $f(x_1)=f(x_2)$, $f(y_1)=f(y_2)$. Since $f$ is a morphism in $\RelMon$,
$$
\begin{tikzcd}
A\times A
	\ar[r,"f\times f"]
	\ar[d,"+"']
&
B\times B
	\ar[d,"+"]
\\
A
	\ar[ru,phantom,scalearrow={0.333}{start}{end}]
	\ar[ru,Rightarrow,to path = (start)--(end)]
	\ar[r,"f"']
&
B
\end{tikzcd}
$$
commutes, so the existence of $x_1+x_2$ in $A$ implies the existence of $f(x_1)+f(x_2)$
in $B$ and $f(x_1+x_2)=f(x_1)+f(x_2)$.
Similarly, $f(y_1+y_2)=f(y_1)+f(y_2)$, so $f(x_1+x_2)=f(y_1+y_2)$, meaning that
$x_1+x_2\sim y_1+y_2$.

Suppose that $x+y$ exists and that $x+y\sim z$, that means, $f(x+y)=f(z)$. By Proposition
\ref{prop:leftadjoint} (L1), there are $x_1,y_1$ such that $f(x_1)=f(x)$, $f(y_1)=f(y)$ and
$x_1+y_1=z$, so (C5) holds.
\end{proof}

\begin{proposition}
Let $\sim$ be a relation on a partial abelian monoid $(A,+,0)$, satisfying the conditions from
Proposition \ref{prop:congruence} and an additional condition
$$
x\sim 0\implies x=0.
$$
Then the quotient map $f\colon A\to A/\sim$ given by $f(x)=[x]\sim$ is a left adjoint in $\RelMon$ and
$\sim=f^\dag\circ f$.
\end{proposition}
\begin{proof}
By \cite{ChePul:SILiPAM}, $f$ is a morphism of partial abelian monoids, hence it is a morphism in $\RelMon$.
The condition (C5) implies (L1) and the additional condition implies (L2).
\end{proof}

Thus, we may say that some of the conditions from the paper \cite{ChePul:SILiPAM} come from
the 2-structure on $\RelMon$.

Finally, let us mention another definition, from the classical paper \cite{loomis1955lattice}.
\begin{definition}
Let $A$ be a complete orthomodular lattice. A {\em dimension equivalence} on $A$
is a equivalence relation on $A$ such that
\begin{enumerate}
\item[(A)]If $a\sim 0$, then $a=0$.
\item[(B)]If $a_1\perp a_2$ and $a_1\vee a_2\sim b$, then there exists an orthogonal
decomposition of $b$, $b=b_1\vee b_2$, such that $b_1\sim a_1$ and $b_2\sim a_2$.
\item[(C)]If $\{a_\alpha\}$ and $\{b_\alpha\}$ are pairwise orthogonal families of elements,
such that $a_\alpha\sim b_\alpha$ for all $\alpha$, then $\bigvee_\alpha a_\alpha=\bigvee_\alpha b_\alpha$.
\item[(D)]If $a$ and $b$ are not orthogonal in $A$ then there are nonzero $a_1,b_1$ in $A$ such that
$a\geq a_1$, $b\geq b_1$ and $a_1\sim b_1$.
\end{enumerate}
\end{definition}

An {\em orthomodular lattice} can be defined as an effect algebra that is
lattice-ordered and satisfies the condition $a\perp a\implies a=0$.
Note that (A) is (M3), (B) is (M2) and (C) is an infinitary version of (C2).
So a dimensional equivalence on an orthomodular lattice is a particular type of
monad in $\RelMon$ arising from an adjunction. It remains an open problem whether we can obtain the conditions
(C) and (D) using the 2-categorical machinery within $\RelMon$. Especially, the 
condition (D) remains a puzzle to us.

\noindent {\bf Acknowledgements}
{\em We are indebted to both anonymous referees for valuable comments and suggestions.}

%\bibliography{mybib}

\begin{thebibliography}{10}

\bibitem{abramsky2009categorical}
Samson Abramsky and Bob Coecke.
\newblock {C}ategorical {Q}uantum {M}echanics.
\newblock In Kurt Engesser, Dov~M. Gabbay, and Daniel Lehmann, editors, {\em
  Handbook of Quantum Logic and Quantum Structures}, pages 261--32. Elsevier,
  Amsterdam, 2009.

\bibitem{awodey2006category}
Steve Awodey.
\newblock {\em Category {T}heory}.
\newblock Number~49 in Oxford Logic Guides. Oxford University Press, 2006.

\bibitem{barr1970relational}
Michael Barr.
\newblock {\em Relational algebras}, pages 39--55.
\newblock Springer Berlin Heidelberg, Berlin, Heidelberg, 1970.

\bibitem{benabou1967introduction}
Jean B{\'e}nabou.
\newblock Introduction to bicategories.
\newblock In {\em Reports of the Midwest Category Seminar}, pages 1--77.
  Springer, 1967.

\bibitem{BusGraLah:OQP}
P.~Bush, M.~Grabowski, and P.~Lahti.
\newblock {\em Operational Quantum Physics}.
\newblock Springer-Verlag, Berlin, 1995.

\bibitem{benabou1965categories}
J~Bénabou.
\newblock Cat{é}gories relatives.
\newblock {\em C.R. Acad. Sci. Paris}, 260:3824--3827, 1965.

\bibitem{ChePul:SILiPAM}
G.~Chevalier and S.Pulmannov{\'a}.
\newblock Some ideal lattices in partial abelian monoids and effect algebras.
\newblock {\em Order}, 17:75--92, 2000.

\bibitem{contreras2015groupoids}
Ivan Contreras.
\newblock Groupoids, {F}robenius algebras and {P}oisson sigma models.
\newblock In {\em Mathematical Aspects of Quantum Field Theories}, pages
  413--426. Springer, 2015.

\bibitem{DvuPul:NTiQS}
A.~Dvure{\v c}enskij and S.~Pulmannov{\'a}.
\newblock {\em New Trends in Quantum Structures}.
\newblock Kluwer, Dordrecht and Ister Science, Bratislava, 2000.

\bibitem{ehresmann1963categories}
Charles Ehresmann.
\newblock Cat{é}gories structur{é}es.
\newblock {\em Ann. Sci. École Norm. Sup.}, 80(3):349--426, 1963.

\bibitem{FouBen:EAaUQL}
D.J. Foulis and M.K. Bennett.
\newblock Effect algebras and unsharp quantum logics.
\newblock {\em Found. Phys.}, 24:1325--1346, 1994.

\bibitem{GiuGre:TaFLfUP}
R.~Giuntini and H.~Greuling.
\newblock Toward a formal language for unsharp properties.
\newblock {\em Found. Phys.}, 19:931--945, 1989.

\bibitem{Goo:POAGwI}
K.R. Goodearl.
\newblock {\em Partially ordered abelian groups with interpolation}.
\newblock Amer. Math. Soc, Providence, 1986.

\bibitem{Gra:GLT}
G.~Gr{\" a}tzer.
\newblock {\em General Lattice Theory}.
\newblock Birkh{\" a}user, second edition, 1998.

\bibitem{heunen2013relative}
Chris Heunen, Ivan Contreras, and Alberto~S Cattaneo.
\newblock Relative {F}robenius algebras are groupoids.
\newblock {\em Journal of Pure and Applied Algebra}, 217:114--124, 2013.

\bibitem{heunen2016monads}
Chris Heunen and Martti Karvonen.
\newblock Monads on dagger categories.
\newblock {\em Theory and Applications of Categories}, 31:1016--1043, 2016.

\bibitem{heunen2015categories}
Chris Heunen and Sean Tull.
\newblock Categories of relations as models of quantum theory.
\newblock In {\em Quantum Physics and Logic 2015}, volume 195 of {\em
  Electronic Proceedings in Theoretical Computer Science}, pages 247--261,
  2015.

\bibitem{JenPul:QoPAMatRDP}
G.~Jen{\v c}a and S.~Pulmannov{\'a}.
\newblock Quotients of partial abelian monoids and the {R}iesz decomposition
  property.
\newblock {\em Algebra univ.}, 47:443--477, 2002.

\bibitem{kelly1974review}
Gregory~M Kelly and Ross Street.
\newblock Review of the elements of 2-categories.
\newblock In {\em Category seminar}, pages 75--103. Springer, 1974.

\bibitem{kelly1982basic}
Max Kelly.
\newblock {\em Basic concepts of enriched category theory}, volume~64.
\newblock CUP Archive, 1982.

\bibitem{kenney2011categories}
Toby Kenney and Robert Paré.
\newblock Categories as monoids in {S}pan, {R}el and {S}up.
\newblock {\em Cahiers de topologie et g{\'e}om{\'e}trie diff{\'e}rentielle
  cat{\'e}goriques}, 52(3):209--240, 2011.

\bibitem{KopCho:DP}
F.~K{\^o}pka and F.~Chovanec.
\newblock D-posets.
\newblock {\em Math. Slovaca}, 44:21--34, 1994.

\bibitem{lack20102categories}
Stephen Lack.
\newblock A 2-categories companion.
\newblock In {\em Towards higher categories}, pages 105--191. Springer, 2010.

\bibitem{mac1998categories}
Saunders~Mac Lane.
\newblock {\em Categories for the Working Mathematician}.
\newblock Number~5 in Graduate Texts in Mathematics. Springer-Verlag, 1971.

\bibitem{leinster2004higher}
Tom Leinster.
\newblock {\em Higher operads, higher categories}, volume 298.
\newblock Cambridge University Press, 2004.

\bibitem{loomis1955lattice}
Lynn~H Loomis.
\newblock The lattice theoretic background of the dimension theory of operator
  algebras.
\newblock {\em Memoirs of the AMS}, 18, 1955.

\bibitem{Lud:FoQM}
G.~Ludwig.
\newblock {\em Foundations of Quantum Mechanics}.
\newblock Springer-Verlag, Berlin, 1983.

\bibitem{pavlovic2017modular}
Dusko Pavlovic and Peter-Michael Seidel.
\newblock (modular) effect algebras are equivalent to ({F}robenius) antispecial
  algebras.
\newblock In Ross Duncan and Chris Heunen, editors, {\em {\rm Proceedings 13th
  International Conference on} Quantum Physics and Logic, {\rm Glasgow,
  Scotland, 6-10 June 2016}}, volume 236 of {\em Electronic Proceedings in
  Theoretical Computer Science}, pages 145--160. Open Publishing Association,
  2017.

\bibitem{street1972formal}
Ross Street.
\newblock The formal theory of monads.
\newblock {\em Journal of Pure and Applied Algebra}, 2(2):149--168, 1972.

\bibitem{wall1937hypergroups}
H.~S. Wall.
\newblock Hypergroups.
\newblock {\em American Journal of Mathematics}, 59(1):77--98, 1937.

\end{thebibliography}
%\bibliographystyle{plain}

\end{document}